\documentclass[a4paper,12pt]{article}
\usepackage{amssymb}
\usepackage{amsthm}
\usepackage[dvips]{graphicx}
\usepackage{amsmath}
\usepackage{fancyhdr}
\usepackage[all]{xy}

\theoremstyle{plain}
\newtheorem{theorem}{Theorem}[section]

\theoremstyle{remark}
\newtheorem{remark}[theorem]{\sc Remark}
\newtheorem{example}[theorem]{\sc Example}
\newtheorem*{acknowledgements}{\sc Acknowledgements}

\renewcommand{\baselinestretch}{1.08}

\newenvironment{reference}[1]{%
\renewcommand{\labelenumi}{[\arabic{enumi}]}
\begin{flushleft}\normalsize{\textsc{References}}\end{flushleft}%
\begin{enumerate}\setlength{\itemsep}{-5pt}\small
}{\end{enumerate}}

\makeatletter
\@addtoreset{equation}{section}
\makeatother

\makeatletter
\renewcommand{\section}{%
\@startsection{section}{1}{\z@}%
{3.5ex \@plus -1ex \@minus -.2ex}%
{2.3ex \@plus.2ex}%
{\reset@font\normalsize\scshape}}

\makeatother

\makeatletter
\renewcommand{\subsection}{%
\@startsection{subsection}{2}{\z@}%
{-3.5ex \@plus -1ex \@minus -.2ex}%
{-2.3ex \@plus.2ex}%
{\reset@font\normalsize\scshape}}
\makeatother

\usepackage[OT2,T1]{fontenc}
\DeclareSymbolFont{cyrletters}{OT2}{wncyr}{m}{n}
\DeclareMathSymbol{\Sha}{\mathalpha}{cyrletters}{"58}

\title{\vspace*{-22mm}
\large{\textbf{
On $p$-class groups of relative cyclic $p$-extensions
\\
}}
}

\author{
\textsc{\normalsize Yasushi Mizusawa}
\and
\textsc{\normalsize Kota Yamamoto}
}

\date{}

\begin{document}
{\renewcommand{\baselinestretch}{1.05} \maketitle}

\vspace*{-11mm}
\renewcommand{\abstractname}{}
{\renewcommand{\baselinestretch}{1.05}
\begin{abstract}{\small 
\noindent\textsc{Abstract.} 
We prove a general stability theorem for $p$-class groups of number fields along relative cyclic extensions of degree $p^2$, 
which is a generalization of a finite-extension version of Fukuda's theorem by Li, Ouyang, Xu and Zhang. 
As an application, we give an example of pseudo-null Iwasawa module over a certain $2$-adic Lie extension. 
}\end{abstract}}

\section{Introduction}

Let $p$ be a fixed prime number. 
For a finite extension $k$ of the rational number field $\mathbb Q$, 
we denote by $A(k)$ be the Sylow $p$-subgroup of the ideal class group $Cl(k)$ of $k$. 
As seen in the ambiguous class number formula (see \cite{Lem13,Yok67} etc.) and Iwasawa theory (see \cite{Gre76,Iwa56,Iwa59} etc.), 
$p$-divisibility and $p$-stability of class numbers in relative cyclic $p$-extensions 
are typical subjects of algebraic number theory. 
In particular, 
the following theorem provides many explicit examples of finite Iwasawa modules. 

\begin{theorem}[{Fukuda \cite{Fuk94}}]\label{thm:Fukuda}
Let $k_{\infty}/k$ be a $\mathbb Z_p$-extension (i.e., an infinite procyclic pro-$p$-extension) which is totally ramified at any ramified primes, 
and let $k_n/k$ be the subextension of degree $p^n$. 
If $|A(k_1)|=|A(k)|$, then $A(k_n) \simeq A(k)$ for all positive integer $n$. 
\end{theorem}

The proof of Theorem \ref{thm:Fukuda} is based on the theory of modules over Iwasawa algebra. 
By the same proof, 
we also obtain a $p$-rank version (\cite[Theorem 1 (2)]{Fuk94}), a version for $p$-ideal class groups (\cite[Proposition 3]{Miz10}) and a $p^e$-rank version for ray class groups (\cite[Theorem 4.3]{Miz18}) etc. 
Moreover, 
by a parallel proof based on Galois module theory, 
a version of Theorem \ref{thm:Fukuda} for cyclic extension of degree $p^2$ 
has been obtained as follows. 

\begin{theorem}[{Li, Ouyang, Xu and Zhang \cite{LOXZ}}]\label{thm:LOXZ}
Let $k''/k$ be a cyclic extension of degree $p^2$ 
with the subextension $k'/k$ of degree $p$. 
Assume that $k''/k$ is totally ramified at any ramified primes, and not unramified. 
Then, for any positive integer $e$, 
$A(k'')/p^eA(k'') \simeq A(k')/p^eA(k') \simeq A(k)/p^eA(k)$ 
if $|A(k')/p^eA(k')|=|A(k)/p^eA(k)|$. 
\end{theorem}

On the other hand, 
by arguments similar to the proof of Iwasawa's $p$-divisibility theorem (\cite{Iwa56}, see also \cite[Theorem 10.4]{Was}), 
Theorem \ref{thm:LOXZ} for $k$ with cyclic $A(k)$ is also obtained 
under more general ramification condition (see \cite[Theorem 2.1 and Corollary 2.3]{Kota20}, also \cite[Proposition 1]{C4MY}). 
In this paper, combining the ideas of these results, 
we give a generalization of Theorem \ref{thm:LOXZ} in more general situation 
where the number fields are not necessarily finite extensions of $\mathbb Q$. 
Since our proof is not based on Galois module theory, 
another proof of Theorem \ref{thm:LOXZ} is also obtained.

\section{Main theorem}

We denote by $\mathbb N$ the set of positive integers. 
Let $K$ be an algebraic extension of $\mathbb Q$, 
and let $S$, $T$ be sets of primes of a subextension of $K/\mathbb Q$ such that $S \cap T=\emptyset$. 
We denote by $L_{S,\infty}^T(K)$ the maximal abelian pro-$p$-extension of $K$ 
which is unramified outside $S$ and totally decomposed at any primes lying over $T$. 
For each $e \in \mathbb N$, let $L_{S,e}^T(K)/K$ be the maximal subextension of $L_{S,\infty}^T(K)/K$ such that the exponent of the Galois group is at most $p^e$. 
Put $A_{S,e}^T(K)=\mathrm{Gal}(L_{S,e}^T(K)/K)$ for $e \in \mathbb N \cup \{\infty\}$. 
Then $L_{S,\infty}^T(K)=\bigcup_{e \in \mathbb N} L_{S,e}^T(K)$, and $A_{S,e}^T(K) \simeq A_{S,\infty}^T(K)/p^eA_{S,\infty}^T(K)$ for each $e \in \mathbb N$. 

If $K/\mathbb Q$ is a finite extension, 
$A_{\emptyset,e}^{\emptyset}(K) \simeq A_{\emptyset,\infty}^{\emptyset}(K) \simeq A(K)$ for all sufficiently large $e \in \mathbb N$. 
Moreover if $S$ is a finite set, 
$A_{S,\infty}^T(K)$ is finitely generated as an abelian pro-$p$ group. 
Then, for each $e \in \mathbb N$, 
$A_{S,e}^T(K)$ is isomorphic to the quotient $p$-group of the ray class group of $K$ modulo sufficiently high power of $\prod_{v \in S}v$ factored by the minimal subgroup containing all $p^e$th power elements and all classes of primes lying over $T$. 

The main result of this paper is the following theorem. 

\begin{theorem}\label{thm:main}
Let $K$ be an algebraic extension of $\mathbb Q$, and 
let $\varSigma$, $S$, $T$ be sets of primes of $K$ such that $S \cap T =\emptyset$. 
Suppose that $K''/K$ is a cyclic extension of degree $p^2$ unramified outside $\varSigma$ with the unique subextension $K'/K$ of degree $p$. 
Assume that the following conditions are satisfied: 
\begin{enumerate}
\renewcommand{\theenumi}{{\rm (\arabic{enumi})}}
\renewcommand{\labelenumi}{{\rm (\arabic{enumi})}}

\item\label{asm:S-S+T}
$(\varSigma \setminus S) \cup T \neq \emptyset$. 

\item\label{asm:totram}
$K''/K$ is totally ramified at any primes lying over $\varSigma \setminus (S \cup T)$. 


\item\label{asm:nonsp}
No prime lying over $T$ splits in $K'/K$. 

\end{enumerate}
Then the following statements hold true for each $e \in \mathbb N \cup \{\infty\}$: 
\begin{enumerate}
\item[$\cdot$]
The restriction mapping $\rho_e : A_{S,e}^T(K') \rightarrow A_{S,e}^T(K)$ is surjective. 
\item[$\cdot$]
If $\rho_e$ is an isomorphism, 
$A_{S,e}^T(K'') \simeq A_{S,e}^T(K') \simeq A_{S,e}^T(K)$ via the restriction mappings. 
Then moreover $A_{S,e}^{\varSigma \cap T}(K'') \simeq A_{S,e}^T(K'')$ 
if $\varSigma \setminus S \neq \emptyset$. 
\end{enumerate}
\end{theorem}

\begin{remark}
Suppose that $K/\mathbb Q$ is a finite extension. 
If $S=T=\emptyset$, 
Theorem \ref{thm:main} is exactly Theorem \ref{thm:LOXZ}. 
Theorem \ref{thm:Fukuda} is obtained by the recursive use of Theorem \ref{thm:LOXZ} for $k_{n+2}/k_n$. 
If $S=\emptyset$, $T \subset \varSigma$, and $A_{\emptyset,\infty}^T(K)$ is trivial, then 
Theorem \ref{thm:main} is just \cite[Theorem 2.1]{Kota20} and \cite[Proposition 1]{C4MY}. 
\end{remark}

\section{Proof of Theorem \ref{thm:main}}

By \ref{asm:totram} and \ref{asm:nonsp}, $(\varSigma \setminus S) \cup T$ contains no archimedean primes. 
Put $L=L_{S,e}^T(K)$, $L'=L_{S,e}^T(K')$ and 
\[
L''= \left\{\begin{array}{ll}
L_{S,e}^{\varSigma \cap T}(K'') & \hbox{if $\varSigma \setminus S \neq \emptyset$,} \\
L_{S,e}^T(K'') & \hbox{if $\varSigma \setminus S=\emptyset$.} 
\end{array}\right.
\]
Then $LK' \subset L'$ and $LK'' \subset L'K'' \subset L_{S,e}^T(K'') \subset L''$. 
Moreover, $K'' \cap L =K' \cap L =K$ 
by \ref{asm:S-S+T} and \ref{asm:totram} if $T=\emptyset$, 
and by \ref{asm:nonsp} if $T \neq \emptyset$. 
Hence $\rho_e$ is surjective and $\mathrm{Ker}\,\rho_e=\mathrm{Gal}(L'/LK')$. 
Assume that $\rho_e$ is an isomorphism from now on. 
Then $L'=LK'$. 
By the maximality of $L''$, $L''/K$ is a Galois extension, 
which is unramified outside $\varSigma \cup S$. 
Put $G=\mathrm{Gal}(L''/K)$ and $H=\mathrm{Gal}(L''/LK'')$. 
Then $[G,G] \subset H$, where the bracket denotes the closed commutator subgroup. 

Suppose that $H \neq \{1\}$. 
Since $[G,G] \neq [[G,G],G]$ if $[G,G] \neq \{1\}$, 
we have $H \neq [[G,G],G]$. 
Then $H$ has a normal subgroup $N$ of index $|H/N|=p$ such that 
\begin{enumerate}
\item[$\cdot$]
$[G,G] \subset N$ if $H \neq [G,G]$, 
\item[$\cdot$]
$[[G,G],G] \subset N$ if $H=[G,G]$. 
\end{enumerate}
In either case, $[H,G] \subset N$, 
and $N/[H,G]$ is a normal subgroup of $G/[H,G]$. 
Then $N$ is a normal subgroup of $G$, 
and $H/N=\mathrm{Gal}((L'')^N/LK'')$ is contained in the center of $G/N=\mathrm{Gal}((L'')^N/K)$. 
In particular, 
$\mathrm{Gal}((L'')^N/LK'')$ is also contained in the center of $\mathrm{Gal}((L'')^N/L)$. 
Since $\mathrm{Gal}(LK''/L)$ is cyclic, 
$(L'')^N/L$ is an abelian extension of degree $p^3$. 

By \ref{asm:S-S+T}, 
at least one of the following conditions is satisfied: 
\begin{enumerate}
\renewcommand{\theenumi}{{\rm (\roman{enumi})}}
\renewcommand{\labelenumi}{{\rm (\roman{enumi})}}
\item\label{cas:S-(S+T)}
$\varSigma \setminus S \neq \emptyset$ and there is some $v_0 \in \varSigma \setminus (S \cup T)$, 
\item\label{cas:S&T}
$\varSigma \setminus S \neq \emptyset$ and there is some $v_0 \in \varSigma \cap T$, 
\item\label{cas:T}
$\varSigma \setminus S =\emptyset$ and there is some $v_0 \in T$. 
\end{enumerate}
If \ref{cas:S-(S+T)}, 
$(L'')^N/LK''$ is unramified over $v_0$, 
but $LK''/L$ is totally ramified at primes lying over $v_0$ by \ref{asm:totram}. 
If \ref{cas:S&T} or \ref{cas:T}, 
$(L'')^N/LK''$ is decomposed at any primes lying over $v_0$, 
but no primes lying over $v_0$ split in $LK''/L$ by \ref{asm:nonsp}. 
In either case, 
$(L'')^N/L$ is never cyclic, 
and hence $\mathrm{Gal}((L'')^N/L)$ is an abelian group of type $[p^2,p]$. 
There exists uniquely an intermediate field $M$ of $(L'')^N/L$ such that $\mathrm{Gal}(M/L)$ is an abelian group of type $[p,p]$. 
Then $L' \subset M$ and $[M:L']=p$. 
Thus we obtain the following diagram. 
\[
\entrymodifiers={+!!<0pt,\fontdimen22\textfont2>}
\xymatrix{
K'' \ar@{--}[r] & LK'' \ar@{-}[r] & (L'')^N \ar@{--}[r] & L'' \\
K' \ar@{-}[u] \ar@{--}[r] & L' \ar@{-}[u] \ar@{-}[ru]|{\hbox{\,$\cdot$\,}} \ar@{-}[r] & M \ar@{-}[u] & \\
K \ar@{-}[u] \ar@{--}[r] & L \ar@{-}[u] \ar@{-}[ru]|{\hbox{\,$\cdot$\,}} \ar@{-}[r] & \cdot \ar@{-}[u] & \\
} 
\]

Let $v$ be an arbitrary prime of $L$ lying over a prime in $T \setminus \varSigma$. 
Then $v$ is inert in $LK''/L$ by \ref{asm:nonsp}, 
and $(L'')^N/L$ is unramified at $v$. 
Since $(L'')^N/L$ is not cyclic, $(L'')^N/LK''$ is decomposed at any primes lying over $v$. 
Hence $(L'')^N \subset L_{S,e}^T(K'')$ (even if $\varSigma \setminus S \neq \emptyset$). 

Let $w_0$ be an arbitrary prime in $(\varSigma \setminus S) \cup T$. 
There exists such a prime $w_0$ by \ref{asm:S-S+T}. 
For each prime $w$ of $(L'')^N$ lying over $w_0$, 
we denote by $L_w$ either the inertia field or the decomposition field in $(L'')^N/K$ according to whether $w_0 \in \varSigma \setminus (S \cup T)$ or $w_0 \in T$. 
Then $L \subset L_w$. 
Since $(L'')^N \subset L_{S,e}^T(K'')$, 
$(L'')^N/LK''$ is unramified at any primes lying over $w_0$, 
and moreover decomposed at the prime if $w_0 \in T$. 
Therefore $[(L'')^N:L_w] \le p^2$. 
Since $\mathrm{Gal}((L'')^N/L_w)|_{LK''}=\mathrm{Gal}(LK''/L)$ by \ref{asm:totram} and \ref{asm:nonsp}, 
we have $\mathrm{Gal}((L'')^N/L_w) \simeq \mathbb Z/p^2\mathbb Z$. 
In an abelian group $\mathrm{Gal}((L'')^N/L)$ of type $[p^2,p]$, 
all subgroups of order $p^2$ contains a common subgroup $\mathrm{Gal}((L'')^N/M) \simeq \mathbb Z/p\mathbb Z$, 
and $\mathrm{Gal}((L'')^N/L')$ is the unique noncyclic subgroup of order $p^2$. 
Then $L_w \subset M$ and $L_w \not\subset L'$. Hence $K'L_w=L'L_w=M$. 
This implies that $M$ is the common inertia field or the common decomposition field in $(L'')^N/K'$ for all $w|w_0$ 
according to whether $w_0 \in \varSigma \setminus (S \cup T)$ or $w_0 \in T$. 
In particular, $M \subset L_{S,1}^T(L')$. 
For any $\sigma \in \mathrm{Gal}((L'')^N/K')$, $\sigma(w)$ is also lying over $w_0$. 
Since $\sigma(M)$ is the inertia field or the decomposition field of $\sigma(w)$ in $(L'')^N/K'$ according to whether $w_0 \in \varSigma \setminus (S \cup T)$ or $w_0 \in T$, we have $\sigma(M)=M$. 
Hence $M/K'$ is a Galois extension. 
Since $M \subset L_{S,1}^T(L')$, $M/K'$ is unramified outside $S$ and totally decomposed over $T$. 
Since $(L'')^N=MK''$ and $M \cap K''=K'$, 
we have $\mathrm{Gal}(M/K') \simeq \mathrm{Gal}((L'')^N/K'')$. 
Since $\mathrm{Gal}((L'')^N/K'')$ is an abelian pro-$p$ group whose exponent is at most $p^e$, $M \subset L_{S,e}^T(K')=L'$. 
This is a contradiction. 
Therefore $H=\{1\}$, i.e., $LK''=L_{S,e}^T(K'')=L''$. 
Thus we obtain the statements. 

\begin{remark}
The key points of this proof are the topology of pro-$p$ group and the structure of an abelian group of type $[p^2,p]$. 
The use of equivalence between ``$[G,G]=[[G,G],G]$'' and ``$[G,G]=\{1\}$'' corresponds 
to the use of Nakayama's lemma in the original proof of Theorem \ref{thm:Fukuda} and Theorem \ref{thm:LOXZ}. 
\end{remark}

\section{Examples}

There are various examples of Theorem \ref{thm:main} as follows. 

\begin{example}
Suppose $p=3$, and put $\ell=109 \equiv 1 \pmod{p^3}$. 
Put $K''=L_{\{\ell\},\infty}^{\emptyset}(\mathbb Q)$, 
which is a cyclic extension of $\mathbb Q$ of degree $p^3$. 
Suppose that $K'=L_{\{\ell\},2}^{\emptyset}(\mathbb Q)$ and $K=L_{\{\ell\},1}^{\emptyset}(\mathbb Q)$. 
Then $\varSigma =\{v_{\ell}\}$ where $v_{\ell}$ is the unique prime of $K$ lying over $\ell$. 
Let $S=\{v_7,v_{43},v_{43}',v_{43}''\}$ be the set of primes of $K$ lying over $7$ or $43$. 
Put $T=\emptyset$. 
By \cite{PARI}, we see that $A_{S,\infty}^{\emptyset}(K) \simeq A_{S,\infty}^{\emptyset}(K') \simeq \texttt{[9,3]}$. 
Hence $A_{S,\infty}^{\emptyset}(K'') \simeq \mathbb Z/9\mathbb Z \oplus \mathbb Z/3\mathbb Z$ by Theorem \ref{thm:main}. 
\end{example}

\begin{example}
Suppose $p=2$, and put $k''=L_{\{41,3\},\infty}^{\emptyset}(\mathbb Q)$, which is a totally real cyclic octic field. 
Put $k'=L_{\{41\},\infty}^{\emptyset}(\mathbb Q)=L_{\{41,3\},2}^{\emptyset}(\mathbb Q)$ and $k=L_{\{41\},1}^{\emptyset}(\mathbb Q)=\mathbb Q(\sqrt{41})$. 
The prime $3$ is inert in $k'/\mathbb Q$. 
Suppose that $K''=k''K$ and $K'=k'K$ 
where $K=k(\sqrt{7 \cdot 19 \cdot 37})$. 
Then $\varSigma=\{v_3,v_3',v_{41},v_{41}'\}$, 
where $v_3,v_3'$ (resp.\ $v_{41},v_{41}'$) are the distinct primes of $K$ lying over $3$ (resp.\ $41$). 
Suppose $S=\emptyset$, and put $T=\{v_3,v_3'\}$. 
Since $A_{\emptyset,\infty}^T(K) \simeq A(K) \simeq \texttt{[2,2]}$ 
and $A_{\emptyset,\infty}^T(K') \simeq A(K') \simeq \texttt{[2,2]}$ 
by \cite{PARI}, 
we have $A_{\emptyset,\infty}^T(K'') \simeq \mathbb Z/2\mathbb Z \oplus \mathbb Z/2\mathbb Z$ 
by Theorem \ref{thm:main}. 
\end{example}

\begin{example}\label{p2S2Tem}
Suppose $p=2$, and put $K=\mathbb Q(\sqrt{-7})$ whose class number is $1$. 
Put $K^{(\infty)}=K\mathbb Q^{(\infty)}$ and $K^{(n)}=K\mathbb Q^{(n)}$, 
where $\mathbb Q^{(\infty)}$ is the cyclotomic $\mathbb Z_2$-extension of $\mathbb Q$, and $\mathbb Q^{(n)}$ is the subextension of degree $2^n$. 
Suppose that $K''=K^{(2)}$ and $K'=K^{(1)}=K(\sqrt{2})$. 
Then $\varSigma=\{v,v'\}$, where $v$ and $v'$ are the distinct primes lying over $2$. 
Put $S=\{v\}$ and $T=\emptyset$. 
There is an exact sequence 
\[
\entrymodifiers={+!!<0pt,\fontdimen22\textfont2>}
\xymatrix{
\mathcal{O}_{K'}^{\times} \ar[r]^-{\varphi_e} & (\mathcal{O}_{K'}/w^{e+2})^{\times} \ar[r] & Cl_{w^{e+2}}(K') \ar[r] & Cl(K') \ar[r] & 0 
} , 
\]
where $\mathcal{O}_{K'}$ is the ring of algebraic integers in $K'$, 
$w$ is the prime of $K'$ lying over $v$, 
and $Cl_{w^{e+2}}(K')$ is the ray class group of $K'$ modulo $w^{e+2}$. 
Similarly, 
the ray class group of $K$ modulo $v^{e+2}$ is isomorphic to 
$A_{S,e}^{\emptyset}(K) \simeq (\mathbb Z[\frac{1+\sqrt{-7}}{2}]/v^{e+2})^{\times}/\{\pm 1\} \simeq (\mathbb Z/2^{e+2}\mathbb Z)^{\times}/\{\pm 1\}$, 
and hence $A_{S,\infty}^{\emptyset}(K) \simeq \mathbb Z_2$. 
Since $(\mathcal{O}_{K'}/w^{e+2})^{\times} \simeq (\mathbb Z[\sqrt{2}]/\sqrt{2}^{e+2})^{\times}$ and 
$\mathcal{O}_{K'}^{\times}=\mathbb Z[\sqrt{2}]^{\times}$, 
we have 
$\mathrm{Gal}(L_{S,\infty}^{\emptyset}(K')/L_{\emptyset,\infty}^{\emptyset}(K')) \simeq \varprojlim \mathrm{Coker}\,\varphi_e \simeq \mathrm{Gal}(\mathbb Q^{(\infty)}/\mathbb Q(\sqrt{2})) \simeq \mathbb Z_2$. 
Since $Cl(K') \simeq \texttt{[2]}$ and $Cl_{w^5}(K') \simeq \texttt{[4]}$ by \cite{PARI}, 
we see that $A_{S,\infty}^{\emptyset}(K') \simeq \mathbb Z_2$ 
(see e.g.\ \cite[Lemma 4.2]{Kota20}). 
Then $A_{S,\infty}^{\emptyset}(K'') \simeq \mathbb Z_2$ by Theorem \ref{thm:main}. 
By the recursive use of Theorem \ref{thm:main}, 
$A_{S,\infty}^{\emptyset}(K^{(n)}) \simeq \mathbb Z_2$ for all $n \in \mathbb N$. 
\end{example}

Moreover, as an application of Theorem \ref{thm:main}, 
we obtain the following theorem on the Iwasawa modules of $2$-adic Lie iterated extensions. 

\begin{theorem}\label{thm:Z2Z2}
Suppose that $p=2$ and $k=\mathbb Q(\sqrt{-\ell_1\ell_2})$ 
with two distinct prime numbers $\ell_1 \equiv \ell_2 \equiv \pm 5 \pmod{8}$. 
Let $k^{(\infty)}$ be the cyclotomic $\mathbb Z_2$-extension of $k$. 
Suppose that $b_0$ is an algebraic integer in $k$ such that 
$2+b_0 \in \ell_1 (k^{\times})^2 \cup 2\ell_1 (k^{\times})^2$ 
and $4-b_0^2 \in (k^{\times})^2$. 
Let $\{b_{n}\}_{n \in \mathbb N}$ be a sequence satisfying $b_n^2-2=b_{n-1}$ for all $n \in \mathbb N$. 
Put $k_n=k(b_n)$ and $k_{\infty}=\bigcup_{n \in \mathbb N} k_n$. 
Then $k_{\infty}k^{(\infty)}/k$ is a $\mathbb Z_2 \rtimes \mathbb Z_2$-extension unramified outside $2$ 
such that $A_{\emptyset,\infty}^{\{2\}}(k_{\infty}k^{(\infty)}) \simeq \mathbb Z_2$. 
\end{theorem}

\begin{proof}
We denote by $k^{(m)}$ the $m$th layer of the $\mathbb Z_2$-extension $k^{(\infty)}/k$. 
Note that $k^{(1)}=k(\sqrt{2})$ and $k^{(\infty)}/\mathbb Q$ is totally ramified at $2$. 
Put $L=L_{\emptyset,\infty}^{\emptyset}(k^{(\infty)})$ and $L'=L_{\emptyset,\infty}^{\{2\}}(k^{(\infty)})$. 
By the assumption, 
$k_1k^{(1)}=k^{(1)}(\sqrt{2+b_0})=k^{(1)}(\sqrt{\ell_1})=k^{(1)}(\sqrt{-\ell_2})$. 
Then $k_1k^{(\infty)}/k^{(\infty)}$ is an unramified quadratic extension, 
and $k_1k^{(\infty)} \cap L'=k^{(\infty)}$. 
By the assumption and \cite[Theorem 5]{Fer80}, 
$L'/k^{(\infty)}$ is a $\mathbb Z_2$-extension, and $L=k_1L'$. 
In particular, $A_{\emptyset,\infty}^{\{2\}}(k^{(\infty)}) \simeq \mathbb Z_2$. 
Then $L/k_1k^{(\infty)}$ is a $\mathbb Z_2$-extension, 
and $L' \subset L \subset L_{\emptyset,\infty}^{\{2\}}(k_1k^{(\infty)}) \subset \widetilde{L}$, 
where $\widetilde{L}$ denotes the maximal unramified pro-$2$-extension of $k^{(\infty)}$. 
By \cite[Theorem 2.1]{MizB}, $\widetilde{L}/L'$ is a finite cyclic $2$-extension. 
Since a prime lying over $2$ is inert in the quadratic subextension $L/L'$ of $\widetilde{L}/L'$, 
the prime is inert in $\widetilde{L}/L'$. 
Therefore $L = L_{\emptyset,\infty}^{\{2\}}(k_1k^{(\infty)})$, 
and hence $A_{\emptyset,\infty}^{\{2\}}(k_1k^{(\infty)}) \simeq \mathbb Z_2$. 

Put $k_0=k(b_0)=k$. 
By the assumption on $b_0$ and \cite[Theorem 3.1]{Kota20}, 
$k_{n+2}/k_n$ is a cyclic quartic extension unramified outside $2$ for each $n \in \{0\} \cup \mathbb N$. 
By \cite[Proposition 4.1]{Kota20}, $k_{\infty}k^{(\infty)}/k$ is a Galois extension unramified outside $2$. 
Since $k_1k^{(\infty)} \neq k^{(\infty)}$, 
one can see that $k_{n+2}k^{(\infty)}/k_nk^{(\infty)}$ is also a cyclic quartic extension for each $n \in \{0\} \cup \mathbb N$, inductively. 
Then $k_{\infty}k^{(\infty)}/k^{(\infty)}$ is a $\mathbb Z_2$-extension by \cite[Lemma 4.2]{Kota20}, 
and $k_{\infty}k^{(\infty)}/k$ is a $\mathbb Z_2 \rtimes \mathbb Z_2$-extension. 
Recall that the prime $v$ of $k^{(\infty)}$ lying over $2$ is inert in $k_1k^{(\infty)}$. 
Since $L$ contains no cyclic quartic extension of $k^{(\infty)}$ in which $v$ is inert, 
$k_2k^{(\infty)}/k^{(\infty)}$ is not unramified, 
and hence $k_{\infty}k^{(\infty)}/k_1k^{(\infty)}$ is totally ramified at the prime lying over $2$. 
By applying Theorem \ref{thm:main} for $K''=k_2k^{(\infty)}$, $K=k^{(\infty)}$, $\varSigma=T=\{v\}$ and $S=\emptyset$, 
we see that $A_{\emptyset,\infty}^{\{2\}}(k_2k^{(\infty)}) \simeq \mathbb Z_2$. 
By the recursive use of Theorem \ref{thm:main} for the quartic extensions $k_{n+2}k^{(\infty)}/k_nk^{(\infty)}$, 
$A_{\emptyset,\infty}^{\{2\}}(k_nk^{(\infty)}) \simeq \mathbb Z_2$ for all $n \in \mathbb N$, and hence $A_{\emptyset,\infty}^{\{2\}}(k_{\infty}k^{(\infty)}) \simeq \mathbb Z_2$. 
\end{proof}

\begin{remark}
In the situation of Theorem \ref{thm:Z2Z2}, 
the inertia field of the prime lying over $2$ in $k_{\infty}k^{(\infty)}/k$ is $k(\sqrt{\ell_1})$, which is either $k_1=k(\sqrt{2+b_0})$ or $k(\sqrt{2(2+b_0)})$. 
Since the prime $2$ does not split in $k_{\infty}k^{(\infty)}/\mathbb Q$, 
$\mathrm{Ker}(A_{\emptyset,\infty}^{\emptyset}(k_{\infty}k^{(\infty)}) \rightarrow A_{\emptyset,\infty}^{\{2\}}(k_{\infty}k^{(\infty)}))$ is procyclic. 
Then $A_{\emptyset,\infty}^{\emptyset}(k_{\infty}k^{(\infty)})$ is a finitely generated $\mathbb Z_2$-module, in particular a ``pseudo-null'' $\mathbb Z_2[[\mathrm{Gal}(k_{\infty}k^{(\infty)}/k)]]$-module 
(cf. \cite{C4MY,Kota20} etc.). 
\end{remark}

\begin{example}
Suppose that $\ell_1=11$ and $\ell_2=19$ in the situation of Theorem \ref{thm:Z2Z2}. 
Then $\widetilde{L}=L$ by \cite[Corollary 3.4]{MizB}, 
and hence 
$A_{\emptyset,\infty}^{\emptyset}(k_1k^{(\infty)}) \simeq \mathbb Z_2$. 

If $b_0=-93102$, then $2+b_0=-{70}^2\ell_2$ and $4-b_0^2=-{6440}^2 \ell_1\ell_2$. 
Hence $A_{\emptyset,\infty}^{\{2\}}(k_{\infty}k^{(\infty)}) \simeq \mathbb Z_2$ by Theorem \ref{thm:Z2Z2}. 
Since $k(\sqrt{\ell_1})=k_1$, $k_2k^{(\infty)}/k_2$ is totally ramified over $2$. 
By \cite{PARI}, we see that $A(k_2) \simeq \texttt{[2]}$ and $A(k_2k^{(1)}) \simeq \texttt{[8]}$. 
By \cite[Theorem 1 (2)]{Fuk94} (or Theorem \ref{thm:LOXZ}), 
$A_{\emptyset,\infty}^{\emptyset}(k_2k^{(\infty)})$ is procyclic. 
By the recursive use of Theorem \ref{thm:main} for $k_{n+2}k^{(\infty)}/k_nk^{(\infty)}$, 
we have $A_{\emptyset,\infty}^{\emptyset}(k_{\infty}k^{(\infty)}) \simeq \mathbb Z_2$. 

If $b_0=6440\sqrt{-\ell_1\ell_2}$, then $2+b_0=2\ell_1^{-1}(506+35\sqrt{-\ell_1\ell_2})^2$ and 
$4-b_0^2={93102}^2$. 
Hence $A_{\emptyset,\infty}^{\{2\}}(k_{\infty}k^{(\infty)}) \simeq \mathbb Z_2$ by Theorem \ref{thm:Z2Z2}. 
Since $k(\sqrt{\ell_1})=k(\sqrt{2(2+b_0)})$, 
$k_2k^{(\infty)}/k_2k^{(1)}$ is totally ramified over $2$. 
By \cite{PARI}, $A(k_2k^{(1)}) \simeq \texttt{[8]}$ and $A(k_2k^{(2)}) \simeq \texttt{[16]}$ under GRH. 
Then $A_{\emptyset,\infty}^{\emptyset}(k_{\infty}k^{(\infty)}) \simeq \mathbb Z_2$ by the same arguments under GRH. 
\end{example}

\vspace*{15pt}
\begin{acknowledgements}
This work was partially supported by JSPS KAKENHI Grant Number JP17K05167. 
\end{acknowledgements}


\vspace*{10pt}

{\footnotesize 
\noindent
Department of Mathematics, Nagoya Institute of Technology, Gokiso-cho, Showa-ku, Nagoya 466-8555, Japan. 
\ \ \texttt{mizusawa.yasushi@nitech.ac.jp}
\\

\noindent
Division of Mathematics and Mathematical Science, 
Department of Computer Science and Engineering, 
Graduate School of Engineering, 
Nagoya Institute of Technology, Gokiso-cho, Showa-ku, Nagoya 466-8555, Japan. 
\ \ \texttt{k.yamamoto.953@nitech.jp}
\\
}

\end{document}